\documentclass[journal]{IEEEtran}
\ifCLASSINFOpdf
\else
\fi
\usepackage{graphicx}
\usepackage{amsmath}
\usepackage{amssymb}
\usepackage{amsthm}
\usepackage{epstopdf}

\newcommand{\norm}[1]{\left\lVert#1\right\rVert}

\newcounter{tempEquationCounter}
\newcounter{thisEquationNumber}

\def\doubleunderline#1{\underline{\underline{#1}}}
\DeclareMathOperator{\alim}{as-lim}

\newtheorem{lemma}{Lemma}
\newtheorem{theorem}{Theorem}
\newtheorem{definition}{Definition}
\newtheorem{proposition}{Proposition}

\usepackage{color}

\begin{document}
%
% paper title
% Titles are generally capitalized except for words such as a, an, and, as,
% at, but, by, for, in, nor, of, on, or, the, to and up, which are usually
% not capitalized unless they are the first or last word of the title.
% Linebreaks \\ can be used within to get better formatting as desired.
% Do not put math or special symbols in the title.
%\title{Finite-Time Dynamic Watermarking for General LTI Systems}
\title{Sensor Switching Control Under Attacks Detectable by Finite Sample Dynamic Watermarking Tests}
%
%
% author names and IEEE memberships
% note positions of commas and nonbreaking spaces ( ~ ) LaTeX will not break
% a structure at a ~ so this keeps an author's name from being broken across
% two lines.
% use \thanks{} to gain access to the first footnote area
% a separate \thanks must be used for each paragraph as LaTeX2e's \thanks
% was not built to handle multiple paragraphs
%
%This work was supported by a grant from Ford Motor Company via the Ford-UM Alliance under award N022977
%\author{Michael~Shell,~\IEEEmembership{Member,~IEEE,}
%        John~Doe,~\IEEEmembership{Fellow,~OSA,}
%        and~Jane~Doe,~\IEEEmembership{Life~Fellow,~IEEE}% <-this % stops a space
\author{Pedro Hespanhol, Matthew Porter, Ram Vasudevan, and Anil Aswani %
\thanks{This work was supported by the UC Berkeley Center for Long-Term Cybersecurity, and by a grant from Ford Motor Company via the Ford-UM Alliance under award N022977.}% <-this % stops a space
\thanks{Pedro Hespanhol and Anil Aswani are with the Department of Industrial Engineering and Operations Research, University of California, Berkeley, CA 94720, USA 
        {\tt\small pedrohespanhol@berkeley.edu, aaswani@berkeley.edu}}%
				\thanks{Matthew Porter and Ram Vasudevan are with the Department of Mechanical Engineering, University of Michigan, Ann Arbor, MI 48109, USA 
        {\tt\small matthepo@umich.edu, ramv@umich.edu}}%
}        
\maketitle

% As a general rule, do not put math, special symbols or citations
% in the abstract or keywords.
%\mattp{The title, abstract, and introduction seem to emphasize the usage of finite time tests as the central point of the paper, but I could argue that all of the statistical tests in previous works are also finite time tests.}
\begin{abstract}

Control system security is enhanced by the ability
to detect malicious attacks on sensor measurements. Dynamic
watermarking can detect such attacks on linear time-invariant
(LTI) systems. However, existing theory focuses on attack detection
and not on the use of watermarking in conjunction with
attack mitigation strategies. In this paper, we study the problem
of switching between two sets of sensors: One set of sensors has
high accuracy but is vulnerable to attack, while the second set of
sensors has low accuracy but cannot be attacked. The problem is
to design a sensor switching strategy based on attack detection
by dynamic watermarking. This requires new theory because
existing results are not adequate to control or bound the behavior
of sensor switching strategies that use finite data. To overcome
this, we develop new finite sample hypothesis tests for dynamic
watermarking in the case of bounded disturbances, using the
modern theory of concentration of measure for random matrices.
Our resulting switching strategy is validated with a simulation
analysis in an autonomous driving setting, which demonstrates
the strong performance of our proposed policy.
\end{abstract}

% Note that keywords are not normally used for peerreview papers.
\begin{IEEEkeywords}
Dynamic watermarking, observer switching control, finite sample tests
\end{IEEEkeywords}

% For peer review papers, you can put extra information on the cover
% page as needed:
% \ifCLASSOPTIONpeerreview
% \begin{center} \bfseries EDICS Category: 3-BBND \end{center}
% \fi
%
% For peerreview papers, this IEEEtran command inserts a page break and
% creates the second title. It will be ignored for other modes.
\IEEEpeerreviewmaketitle

\section{Introduction}
% The very first letter is a 2 line initial drop letter followed
% by the rest of the first word in caps.
% 
% form to use if the first word consists of a single letter:
% \IEEEPARstart{A}{demo} file is ....
% 
% form to use if you need the single drop letter followed by
% normal text (unknown if ever used by the IEEE):
% \IEEEPARstart{A}{}demo file is ....
% 
% Some journals put the first two words in caps:
% \IEEEPARstart{T}{his demo} file is ....
% 
% Here we have the typical use of a "T" for an initial drop letter
% and "HIS" in caps to complete the first word. %$\IEEEPARstart{T}{his} demo file is intended to serve as a ``starter file''
%for IEEE journal papers produced under \LaTeX\ using
%IEEEtran.cls version 1.8b and later.
% You must have at least 2 lines in the paragraph with the drop letter
% (should never be an issue)
%\hfill mds
%\hfill August 26, 2015
% needed in second column of first page if using \IEEEpubid
%\IEEEpubidadjcol

\label{sec1}

\IEEEPARstart{T}{he} secure and resilient control of cyber-physical systems (CPS) requires safe operation in the face of malicious attacks that can occur on either the physical layer (e.g., sensors and actuators) or the cyber layer (e.g., communication and computation capabilities)\cite{shafi2012cyber}. Real-life incidents like the Maroochy-Shire incident \cite{abrams2008malicious}, the Stuxnet worm \cite{langner2011stuxnet}, and others \cite{cardenas2008research} illustrate the importance of concerns about CPS security. One approach to secure control has been to focus on cybersecurity of CPS \cite{parno2006secure,kumar2006managing,wang2013cyber,kim2012cyber}, but this does not fully exploit the physical aspects of CPS. An alternative is attack identification and detection considering the interplay between the cyber and physical parts of CPS \cite{amin2009safe,cardenas2008research,cardenas2008secure,pasqualetti2013attack}. Many of these techniques are static (i.e., do not consider system dynamics)  \cite{gomez2004power} or passive (i.e., do not actively control system to identify malicious nodes and sensors) \cite{bai2015security,fawzi2014secure,fawzi2011secure}. 

In contrast, \emph{dynamic watermarking} is an active defense technique that injects perturbations into the system control in order to detect attacks \cite{weerakkody2014detecting,mo2009secure,mo2014detecting,gallo2018distributed}. More specifically, this method applies a \emph{private excitation} to the system, which is a disturbance only known to the controller.  Then it uses consistency tests to detect attacks by checking for correlation between sensor measurements and the private excitation. The goal is to be able to detect all sensor attacks whose magnitude exceeds some prespecified amount. %ensure zero-average power for attacks that remain undetected, meaning attacks that remain undetected are necessarily constrained to only a zero power distortion (on average) onto the noise already entering the system. 

\subsection{Asymptotic Results for Dynamic Watermarking}

Research on dynamic watermarking can be divided into two main areas of contribution: The first is the development of statistical hypothesis testing that tries to detect corrupted measurements by observing correlations between sensor outputs and the dynamic watermark \cite{mo2009secure,mo2010false,mo2014detecting,weerakkody2014detecting,mo2015physical}. This set of techniques apply to general LTI systems, but cannot ensure the zero-average-power property for general attack models. The second line of work \cite{satchidanandan2016dynamic,ko2016theory} considers general attack models and develop tests able to ensure that only attacks which add a zero-average-power signal to the sensor measurements can remain undetected, but constrain their analysis to LTI systems with specific structure on their dynamics.

More recently, The work done in \cite{satchidanandan2017minimal} and \cite{hespanhol2017dynamic} attempts to bridge this gap by providing statistical guarantees for complex types of attacks for general LTI systems.
% \mattp{In the previous paragraph you mention that previous work either assumed an attack model or assumed things about the LTI system. Here you say that these two works bridge the gap between these methods but then describe them as doing the same stuff as the previous methods which is confusing.} \pedro{re-phrased it a bit}
While both papers address a general MIMO LTI system, the set of assumptions are somewhat different: the former assumes open-loop stability of the LTI system, and the latter restricts the attack form. In particular, in \cite{hespanhol2017dynamic}, the tests provided are able to detect if a general MIMO LTI system is under a fairly general type of attack. In particular, it considers additive attacks that can dampen/amplify the system measurements, can replay the system from a different initial condition, or can do both. This form of attack, while arguably simple, encompasses many of the types of attacks reported in real-life incidents (e.g., replay attacks \cite{langner2011stuxnet}) as well as compensate for external disturbances not accounted by the system model (e.g., wind when represented via internal model principle \cite{hespanhol2017dynamic}).

We proceed to briefly summarize the results of \cite{hespanhol2017dynamic}, as it is the foundation for this current paper. Consider a MIMO LTI system with partial observations
\begin{equation}
\begin{aligned}
 x_{n+1} = Ax_n + Bu_n + w_n \\
 y_n = Cx_n + z_n + v_n
\end{aligned}
\end{equation}
for some measurement noise $z_n$, system disturbance $w_n$, and attack vector $v_n$. Suppose $(A,B)$ is stabilizable, $(A,C)$ is detectable. Typically, dynamic watermarking approaches will add an additive signal to the control input $u_n = Kx_n + e_n$, where $K$ is some feedback matrix and $e_n$ is our watermarking signal that is unknown to the attacker. Now let $k' = \min\{k\geq 0\ |\ C(A+BK)^kB \neq 0\}$. If we define the test vectors
\begin{equation}
\psi^\top_n = \begin{bmatrix} (C\hat{x}_n^{\vphantom{\top}} - y_n^{\vphantom{\top}})^\top & e_{n-k'-1}^\top\end{bmatrix},
\end{equation}
and the following holds \cite{hespanhol2017dynamic}:
\begin{equation}
\label{eqn:acttestjoint}
\textstyle\alim_{N} \frac{1}{N}\sum_{n=0}^{N-1} \psi_n^{\vphantom{\top}}\psi_n^\top = \begin{bmatrix}C\Sigma_\Delta C^\top+\Sigma_Z & 0\\ 0 & \Sigma_E\end{bmatrix}
\end{equation}
for some specific matrices $\Sigma_\Delta,\Sigma_{E},\Sigma_{Z}$, then all attack vectors $v_n$ following a particular model \cite{hespanhol2017dynamic} are constrained in power
\begin{equation}
\label{eqn:zasp}
\textstyle\alim_N \frac{1}{N}\sum_{n=0}^{N-1} v_n^\top v_n^{\vphantom{\top}} = 0.
\end{equation}
Though these tests only provide asymptotic guarantees, that is enough to construct a statistical version of the test, similar to \cite{satchidanandan2016dynamic} where a hypothesis test is constructed by thresholding the negative log-likelihood. It follows that under a Gaussianity assumption for process and sensor noise, the matrix in (\ref{eqn:acttestjoint}) follows a well-behaved Wishart distribution. While that approach allows us to construct hypothesis tests using known distributions, the dependency of subsequent samples make finite sums display more complex behavior. 
% \mattp{while the matrix might converge to a wishart distribution as the window size is increased it does not follow the distribution for a finite sum due to the dependence of subsequent samples in the summation.} \pedro{Equation (3) is not finite sum but I added a setence clarfying that in the test}
Then it is up to the designer of the watermark to specify a threshold that controls the false error rate. In this framework a rejection of the hypothesis test corresponds to detection of an attack, while an acceptance corresponds to the lack of detection of an attack.  This notation emphasizes the fact that achieving a specified false error rate requires changing the threshold.

\subsection{Intelligent Transportation Systems and Observer Switching}

Though the design and analysis of intelligent transportation systems (ITS) has drawn renewed interest \cite{ko2016theory,gonzalez2010perpetual,aswani2011,zhang2012hierarchical,vasudevan2012safe,mohan2016convex,como2016convexity}, there has been less work on secure control of ITRS. One recent work considered the use of dynamic watermarking to detect sensor attacks in a network of autonomous vehicles coordinated by a supervisory controller\cite{ko2016theory}, while \cite{hespanhol2017statistical} considered a platoon of vehicles where attacks happen not only on the sensors but also on the communication channel.

A particular feature of ITS is the possibility of redundancy in sensing. For instance, one can use a highly accurate satellite-based sensor (susceptible to external attack) and an on-board infrared sensor (\emph{not} susceptible to external attack) in order to obtain spatial data. Then, one way of safeguarding a system susceptible to attacks is to switch from the high accuracy sensor to the on-board sensor when an attack is detected \cite{mitra2016secure}. This approach naturally leads to systems with distributed observers with dynamic switching decision rules \cite{bernat2015multi,mitra2018distributed}. In this scenario, it is crucial to design hypothesis tests that are able to detect attacks while having a decision rule that correctly selects which observer is to be used. Because control switching occurs at finite instances in time, the previous asymptotic results of dynamic watermarking cannot be used for this purpose. The reason, which is subtle, is that hypothesis tests based on characterization of asymptotic distributions will not have the correct theoretical properties in order to ensure proper control of the false alarm rate. Consequently, new finite sample hypothesis tests need to be constructed.

The first contribution of this paper it to provide finite-time guarantees on attack detection via dynamic watermarking, which to the best of our knowledge has not been done before. Namely, we provide statistical tests that provide finite-time guarantees on attack detection, instead of relying of asymptotic behavior of sums of random matrices.  
% \mattp{Be more specific on what these guarantees are because I am still confused. Is this is referencing the finite switching and infinite detection?}\pedro{this is refering to infinite detection. I expanded the phrase} 
We also relate the magnitude of an attack to our test power, by describing the inherent trade-off between the test capability of triggering true detection, and the magnitude of the attacks that are allowed to remain undetected in the long run. The finite sample analysis of dynamic watermarking requires the use of random matrix concentration inequalities, which are useful in analyzing the matrices involved in the evolution of LTI system dynamics. The second major contribution of this paper is to provide finite sample concentration-based tests, which allow us to detect attacks and allow switching decisions based on such tests to correctly report attack detections infinitely often. Namely, if there is no attack, we develop a finite sample test that falsely reports attacks only a finite number of times. This is a crucial feature because it also implies that in the long-run the switching rule based on such a test is correctly selecting which observer is active infinitely often.

\subsection{Outline}

In Sect. \ref{sec2}, we define our notation for this paper. We also present the random matrix concentration inequalities that we use to perform our finite sample analysis. Next, in Sect. \ref{Sec3} we present our general LTI framework with switching observers. Then, we apply the concentration inequalities to the LTI setting in Sect. \ref{Sec4} in order to obtain appropriate concentration for the matrices involved. In Sect. \ref{Sec5} and Sect. \ref{Sec6}, we present the finite sample consistency tests and a simple threshold that relates the attack magnitude to the power of our test. Next, in Sect. \ref{Sec6} we provide some numerical results demonstrating our approach on an autonomous vehicle application.

\section{Preliminaries}

\label{sec2}

In this section, we define all relevant notation concerning the random matrix analysis done throughout the paper. We also define the key concepts of Stein's Method \cite{stein1972bound,mackey2014matrix} applied to matrices and the relevant matrix concentration inequalities that will be used. This method turns out to be key to our finite sample analysis of dynamic watermarking, as it involves analyzing sums of inter-temporal dependent matrices.

\subsection{Notation}

We use the symbol $\norm{\cdot}$ for the spectral norm of a matrix, which is the largest singular value of a general matrix. The space of $d\times d$ Hermitian real-valued matrices is denoted by $\mathcal{H}{^{d}}$. Moreover, the symbols $\lambda_{\max}(A), \lambda_{\min}(A)$ are respectively the maximum and the minimum eigenvalues of an Hermitian matrix $A \in \mathcal{H}^d$. The symbol $\preceq$ refers to the semidefinite partial order, namely $A \preceq B$ if and only if $B - A$ is positive semi-definite (p.s.d). For a matrix $A$, we let $(A)_{ij}$ denote the $(ij)$-th element of $A$. We let $\text{tr}(\cdot)$ do the denote the trace operator.

We also define a master probability space $(\Omega, \mathcal{F}, P)$ and a filtration $\{\mathcal{F}_{k}\}$ contained in the master sigma algebra:
\begin{equation}
\mathcal{F}_{k} \subset \mathcal{F}_{k+1} \text{ and } \mathcal{F}_{k} \subset \mathcal{F} , \forall k \geq 0.
\end{equation}
Given such filtration we also define the conditional expectation $\mathbb{E}_{k}[ \cdot]$. We also let $\epsilon$ denote a Radamacher random variable, that takes values in $\{-1,1\}$ with equal probability. The random matrix concentration inequalities involved in this work are derived based on the method of exchangeable pairs based on the Stein's Method \cite{stein1972bound}. Let $Z$ and $Z'$ be random vectors taking values in a space $\mathbb{R}^d$. We say that $(Z,Z')$ is an \emph{exchangeable pair} if it has the same distribution as $(Z',Z)$. Next, we define a matrix Stein pair:
\begin{definition}
Let $Z$ and $Z'$ be an exchangeable pair of random vectors taking values in a space $\mathcal{Z}$, and let $\psi : \mathcal{Z} \rightarrow \mathcal{H}^{d}$ be a measurable function. Define the random Hermitian matrices
\begin{equation}
X = \psi (Z) \textit{ and } X' = \psi (Z').
\end{equation}
We say that $(X, X')$ is a matrix Stein pair if there is a constant $\beta \in (0,1]$ for which $\mathbb{E}[X - X' | Z] = \beta X \textit{ a.s.}$.
\end{definition}
Note it follows from the above definition that $\mathbb{E}[X] = 0$. Also, $\beta$ is called the \emph{scale factor} of the pair $(X, X')$. 

Lastly, we present the concept of dilations, which are used to derive our results. A symmetric dilation of a real-valued rectangular matrix $B$ is
\begin{equation}
\mathcal{D}(B) = \begin{bmatrix} 0 & B \\ B^{\top} & 0 \end{bmatrix} 
\end{equation}
Note that $\mathcal{D}(B)$ is always symmetric, and it satisfies the following useful property:
\begin{equation}
\mathcal{D}(B)^{2} = \begin{bmatrix} 0 & BB^{\top} \\ B^{\top}B & 0 \end{bmatrix}
\end{equation}
Moreover, observe that the norm of the symmetric dilation has a useful relationship with the norm of the original matrix $\lambda_{\max}(\mathcal{D}(B)) = \|\mathcal{D}(B)\| = \|B\|$. We will construct bounds for symmetric matrices and then we will extend those bounds to non-symmetric matrices by using dilations.

\subsection{Matrix Concentration Inequalities}

In order for us to develop finite sample tests we require matrix concentration inequalities. The random matrices involved in this paper are not independent in the general case. We first present a version of matrix Hoeffding inequality for conditionally independent sums of random matrices, that is random matrices that become independent after conditioning on another matrix. This theorem, and the following theorems about concentrations, were first introduced by \cite{mackey2014matrix}, as generalizations of the (respective) independent cases.
\begin{proposition}
\cite{mackey2014matrix} Consider a finite sequence $(Y_k)_{(k\geq 1)}$ of random matrices in $\mathcal{H}^d$ that are conditionally independent given an auxiliary random matrix $Z$ and finite sequences $(P_k)_{k\geq 1}$ and $(Q_k)_{k\geq 1}$ of deterministic matrices in $\mathcal{H}^d$. Assume that
\begin{equation}
\mathbb{E}[Y_k | Z] = 0 \textit{, }  Y^{2}_{k}  \preceq P^{2}_{k} \textit{, } \mathbb{E}[Y^2_{k}|(Y_{j})_{j \neq k}] \preceq  Q^{2}_{k}\textit{ a.s.} \forall k,
\end{equation}
then for all $t \geq 0$ we have
\begin{equation}
\mathrm{P} \left( \lambda_{\max} \left( \sum_{k=0} Y_k \right) \geq t \right) \leq d \cdot e^{-t^2 / 2\sigma^2}
\end{equation}
where $\sigma^2 = \frac{1}{2} \| \sum_{k} P^{2}_{k} + Q^{2}_{k}) \|$.
\end{proposition}

Next we present a version of the McDiarmid inequality for self-reproducing random matrices.
\begin{proposition}
\cite{mackey2014matrix} Let $z = (Z_1, ..., Z_n)$ be a random vector taking values in a space $\mathcal{Z}$, and, for each index k, let $Z'_k$ and $Z_k$ be conditionally  i.i.d. given $(Z_j)_{j \neq k }$. Suppose that $H: \mathcal{Z} \rightarrow \mathcal{H}^d$ is a function that satisfies the self-reproducing property
\begin{equation}
\label{eqn:onewiths}
\sum_{k=1}^{n} (H(z) - \mathbb{E}[H(z) | (Z_{j})_{j\neq k}]) = s \cdot (H(z) - \mathbb{E}[H(z)] ) \text{ a.s. }
\end{equation}
% \mattp{As defined, $H$ operates on the domain $\mathcal{Z}$ but here you are using it as an operator on a sequence in $\mathcal{Z}$. Is $z$ in the above equation supposed to be $Z_j$ with the sum over $j$?}\pedro{fixed that}
for a parameter $s > 0$, as well as the bounded difference property
\begin{equation}
\mathbb{E}[(H(z) - H(Z_{1}, ..., Z_{k}^{'}, ..., Z_{n}))^{2} | z ] \preceq P^{2}_{k}
\end{equation}
for each index k a.s., where $P_k$ is a deterministic matrix in $\mathcal{H}^d$. Then, for all $t \geq 0$,
\begin{equation}
\mathrm{P}\{\lambda_{\max}(H(z) - \mathbb{E}[H(z)]) \geq t\} \leq d \cdot e^{-st^2 / L}
\end{equation}
for $L = \norm{ \sum_{k=1}^{n} P^{2}_{k} } $. 

%Furthermore,
%\begin{equation}
%\mathbb{E}[\lambda_{max}(H(z) - \mathbb{E}[H(z)])] \leq \sqrt{\frac{L\log{d}}{s}}.
%\end{equation}

\end{proposition}

Now we provide an essential property that is called symmetrization, which is a generalization for summation of the symmetrization property presented in \cite{mackey2014matrix} for a single matrix:
\begin{lemma}
Let $\{X_i\}_{i=1}^{n}$ be a sequence of random Hermitian matrices with $\mathbb{E}[X_i] = 0$. Then
\begin{equation}
\mathbb{E}\left[\textup{tr}\left( \textup{e}^{\sum_{i=1}^{n}X_{i}}\right)\right] \leq \mathbb{E}\left[\textup{tr}\left( \textup{e}^{2\sum_{i=1}^{n}\epsilon_i X_{i}}\right)\right]
\end{equation}
where $\{\epsilon_i\}_{i=1}^{n}$ are i.i.d. Radamacher random variables.
\end{lemma}

\begin{proof}
First, we construct a sequence of copies $\{X'_i\}_{i=1}^{n}$ independent from  $\{X_i\}_{i=1}^{n}$, and let $\mathbb{E}'$ denote the expectation with respect to $\{X'_i\}_{i=1}^{n}$. So we have
\begin{multline}
 \mathbb{E}\left[\text{tr}\left( \textup{e}^{\sum_{i=1}^{n}X_{i}}\right)\right] = \mathbb{E}\left[\text{tr}\left( \textup{e}^{\sum_{i=1}^{n}X_{i} - \mathbb{E}'[X'_{i}]}\right)\right] \leq  \\
 \mathbb{E}\left[\text{tr}\left( \textup{e}^{\sum_{i=1}^{n}X_{i} - X'_{i}}\right)\right] = \mathbb{E}\left[\text{tr}\left( \textup{e}^{\sum_{i=1}^{n}\epsilon_{i}(X_{i} - X'_{i})}\right)\right] 
\end{multline}
where we have sequentially used Jensen's inequality and then the symmetry of $(X_i - X'_i)$. Now we finish the proof by noting
\begin{multline}
 \mathbb{E}\left[\text{tr}\left( \textup{e}^{\sum_{i=1}^{n}X_{i}}\right)\right] \leq \mathbb{E}\left[\text{tr}\left( \textup{e}^{\sum_{i=1}^{n}\epsilon_{i}(X_{i} - X'_{i})}\right)\right] \leq  \\
 \mathbb{E}\left[\text{tr}\left( \text{e}^{\sum_{i=1}^{n}\epsilon_{i}X_{i}} \text{e}^{-\sum_{i=1}^{n}\epsilon_{i}X'_{i}}\right)\right] \leq \\
 \mathbb{E}\left[\text{tr}\left( \text{e}^{2\sum_{i=1}^{n}\epsilon_{i}X_{i}}\right)^{1/2} \text{tr}\left(\text{e}^{-2\sum_{i=1}^{n}\epsilon_{i}X'_{i}}\right)^{1/2}\right] =\\
 \mathbb{E}\left[\text{tr}\left( \textup{e}^{2\sum_{i=1}^{n}\epsilon_i X_{i}}\right)\right]
\end{multline}
where we have sequentially used the Golden-Thompson inequality, the Cauchy-Schwartz inequality two times, and the fact that both factors are identically distributed. (See \cite{bhatia2013matrix} for the definition of those properties.)
\end{proof}

\section{LTI System with Switching}

\label{Sec3}

We consider a MIMO LTI system that allows the controller to switch between two sets of sensors, and we will assume that both the measurement and process noise have stochastic distributions with a bounded support. Namely, we will assume that the noise vectors have bounded norm almost surely. 

\subsection{LTI Formulation}

Consider a MIMO LTI system with partial observations and switching in the sensing
\begin{equation} \label{LTI-SYS}
\begin{aligned}
x_{n+1} &= Ax_n + Bu_n + w_n \\
y_n &= C(\alpha_{n}) x_n + z_n(\alpha_n) + \alpha_n v_n \\
\end{aligned} 
\end{equation}
where $x \in\mathbb{R}^p$, $u\in\mathbb{R}^q$, $y,z,v\in\mathbb{R}^m$, and $\alpha_n\in\{0,1\}$. The $w_n$ represents zero mean i.i.d. process noise with covariance $\Sigma_W$. Moreover, we have
\begin{equation}
\begin{aligned}
&C_n = C(\alpha_n) = \alpha_n C_{1} +  (1-\alpha_n)C_{2}\\
&z_n(\alpha_n) = \alpha_n \zeta_{n} + (1 - \alpha_n) \eta_{n}
\end{aligned}
\end{equation}
where $\zeta_{n}$ and $\eta_{n}$ represent zero mean i.i.d. measurement noise with covariance matrices $\Sigma_{\zeta} \preceq \Sigma_{\eta}$, respectively. Note that $\alpha_n \in \{0, 1\}$ should be interpreted as the switching control action that selects between the observability matrices $C_1$ or $C_2$. The $v_n$ is as an additive measurement disturbance added by an attacker, which can only affect the observations made when the mode $\alpha = 1$ is selected. The idea of this model is that $C_1$ corresponds to a more accurate set of sensors than $C_2$, but conversely that some subset of sensors within $C_1$ are susceptible to an attack whereas the set of all sensors within $C_2$ are \emph{not} susceptible to an attack. We further assume the process noise is independent of the measurement noise, that is $w_n$ for $n\geq0$ is independent of $\zeta_{n},\eta_n$ for $n\geq0$. Lastly we assume both measurement and disturbance noises are bounded in magnitude. Namely, we assume that both measurement noise and systems disturbances are given by i.i.d. bounded random vectors: $\norm{w_k} \leq K_{w}$ and $\norm{z_k} \leq K_{z},\forall k \geq 0$. %The intuition behind this model is to capture the situation where our dynamic system is equipped with two sensors, one that produces high-quality estimates but it is susceptible to attacks, and another that it is not attackable but produces more unreliable measurements.

If $(A,B)$ is stabilizable and both $(A,C_1)$ and $(A,C_2)$ are detectable, then an output-feedback controller can be designed when $v_n\equiv 0$ using an observer and the separation principle.  Let $K$ be a constant state-feedback gain matrix such that $A+BK$ is Schur stable, and let $L_{i}$ be a constant observer gain matrix such that $A+L_{i}C_{i}$ is Schur stable for $i \in \{1, 2\}$. The idea of dynamic watermarking in this context will be to superimpose a private (and random) excitation signal $e_n$ known in value to the controller but unknown in value to the attacker.  As a result, we will apply the control input $u_n = Kx'_n + e_n$, where $x'_n$ is the observer-estimated state and $e_n$ are i.i.d. random vectors on a bounded support, such that $\norm{e_k} \leq K_e, \forall k \geq 0$, with zero mean and constant variance $\Sigma_E$ fixed by the controller. Let
\begin{equation}
\begin{aligned}
&L(\alpha) =\alpha L_{1} +  (1-\alpha)L_{2}\\
&L_n = L(\alpha_n)\\
&\underline{L(\alpha)}^\top = \begin{bmatrix} 0 & -L(\alpha)^\top\end{bmatrix}
\end{aligned}
\end{equation}
Moreover, let $\tilde{x}^\top = \begin{bmatrix} x^\top & x'^\top\end{bmatrix}$, and define:
\begin{equation}
\begin{aligned}
\underline{B}^\top &= \begin{bmatrix} B^\top & B^\top\end{bmatrix} \text{, } \underline{D}^\top = \begin{bmatrix} \mathbb{I} & 0\end{bmatrix}\text{, and} \\ 
\underline{A}(\alpha)&= \begin{bmatrix} A & BK \\ -L(\alpha)C(\alpha) & A+BK+L(\alpha)C(\alpha)\end{bmatrix}.
\end{aligned}
\end{equation}
Then the closed-loop system with private excitation is given by:
\begin{equation}
\tilde{x}_{n+1} = \underline{A}(\alpha_{n})\tilde{x}_n + \underline{B}e_n + \underline{D}w_n + \underline{L}(\alpha_{n})(z_n(\alpha_{n})+\alpha_{n} v_n).
\end{equation}
If we define the observation error $\delta' = x'-x$, then with the change of variables $\check{x}^\top = \begin{bmatrix} x^\top & \delta'^\top\end{bmatrix}$ we have the dynamics
\begin{equation}
\check{x}_{n+1} = \doubleunderline{A}(\alpha_{n})\check{x}_n + \doubleunderline{B}e_n + \doubleunderline{D}w_n + \doubleunderline{L}(\alpha_{n})(z_n(\alpha)+\alpha v_n)
\end{equation}
where we further define the following matrices
\begin{equation}
\begin{aligned}
\doubleunderline{B}^\top = \begin{bmatrix}B^\top & 0\end{bmatrix} \text{, } \doubleunderline{D}^\top = \begin{bmatrix} \mathbb{I} & -\mathbb{I}\end{bmatrix} \text{, } \doubleunderline{L}(\alpha) = \underline{L}(\alpha), \\
\text{and }\doubleunderline{A}(\alpha) = \begin{bmatrix} A+BK & BK \\ 0 & A+L(\alpha)C(\alpha)\end{bmatrix}.
\end{aligned}
\end{equation}
Recall that $\doubleunderline{A}(\alpha)$ is Schur stable whenever $A+BK$ and $A+L(\alpha)C(\alpha)$ are both Schur stable.

There is one technical point that needs to be addressed before proceeding: Since there is switching between observers, the closed-loop system will not necessarily be stable even though $A+BK$ and $A+L(\alpha)C(\alpha)$ are both Schur stable. One approach to resolving this issue is limiting the rate of switching, as follows:

\begin{proposition}
\label{prop:swr}
let $P$ be the solution of the Lyapunov equation 
\begin{equation}
\doubleunderline{A}(1)P\doubleunderline{A}(1)^\top - P = -\mathbb{I},
\end{equation}
where $\mathbb{I}$ is the identity matrix. Let $\tau$ be the smallest positive integer such that
\begin{equation}
\doubleunderline{A}(0)^\tau P(\doubleunderline{A}(0)^\tau)^\top - P \leq -\mathbb{I}.
\end{equation}
Then the the closed-loop system is stable under switching policies where: whenever we switch from $\alpha = 1$ to $\alpha = 0$ we maintain $\alpha = 0$ for at least $\tau$ time steps before any possible switching occurs to $\alpha = 1$.
\end{proposition}
Lastly, we note that such a $\tau$ exists because $\doubleunderline{A}(0)$ is Schur stable. 

\section{Matrix Inequalities for General LTI Systems}

\label{Sec4}

We will now apply the abstract concentration inequalities presented in Sect. \ref{sec2} to our LTI setting with switching. We will begin our analysis consider that the system is under no attack.  Under no attack we would like to keep using the most accurate sensor -- that is keeping our switching control $\alpha_n \equiv 1$ for all $n\geq 0$. However, as it is usually observed for any kind of tests based on random quantities, we are susceptible to commit what is commonly known as false positive or type I errors. Hence our goal is to provide finite sample tests based on matrix concentration of measure such that type I errors happen only a finite number of times throughout the evolution of the system. This would imply that those tests report correctly that there is no attack infinitely often. To that end, we will utilize two observers: The first observer obtain system measurements from the switched system, using $C(\alpha_n)$; The second observer never switches and keeps measuring the system using the vulnerable sensor, using $C_1$. The finite-time statistical tests and the concentration inequalities analysis presented in this section are referring to quantities associated with the second observer. For ease of notation and presentation we drop the subscript of the analysis define $C = C_1$ and $L=L_1$. Moreover, for the second observer we define: $\hat{x}_{n}$ and $\delta_{n}$ to denote the estimate state and observation error:
\begin{equation}  \label{xhat}
\hat{x}_{n+1} = (A + BK) \hat{x}_{n} + LC(\hat{x}_{n} - x_n) + Be_n - L z_{n}
\end{equation}
and $\delta_n = \hat{x}_{n} - x_{n}$. Then by the same type of variable substitution:
\begin{equation} \label{deltahat}
\delta_{n+1} = (A+LC)\delta_{n} - w_{n} + -L z_n
\end{equation}
We will start by bounding the vector $C\delta_n -z_n$:
\begin{theorem}
Let $\delta_n = \hat{x}_{n} - x_{n}$. Assume that both measurement noise and systems disturbances are given by i.i.d. bounded random vectors: $\norm{w_k} \leq K_{w} \text{ and } \norm{z_k} \leq K_{z} , \forall k \geq 0$. Then when $v_n\equiv 0$ for all $n\geq 0$ we have
\begin{equation}
\|C\delta_{n} - z_n\| \leq \bar{K}_{n}
\end{equation}
where $\bar{K}_{n} = K_z + \sum_{k=0}^{n-1} \norm{C\bar{D}_{k}} K_{w} + \norm{C\bar{L}_{k}}K_{z}$
and
\begin{equation}
(C\delta_{n} - z_{n})(C\delta_{n} - z_{n})^{\top} \preceq \bar{K}^{2}_{n}\mathbb{I}.
\end{equation}
Moreover, it follows that
\begin{equation}
\begin{aligned}
 \mathbb{E}[(C\delta_{n}& - z_n)(C\delta_{n} - z_n)^{\top}]= \\
&C\left(\sum_{k=0}^{n-1} \bar{D}_{k}\Sigma_{w}\bar{D}^{\top}_{k} + \bar{L}_{k}\Sigma_{z}\bar{L}^{\top}_{k}\right) C^{\top} + \Sigma_{z}
\end{aligned}
\end{equation}
where
\begin{equation}
\begin{aligned}
    \bar{D}_{k} &=-(A+LC)^{n-1-k}\\
    \bar{L}_{k} &=-(A+LC)^{n-1-k}L^{\top}.
\end{aligned}
\end{equation}
\end{theorem}

\begin{proof}
Recall our definition of $\delta_n$ (Eq. \ref{deltahat}) we can write:
%\begin{equation}
%hat{x}_{n+1} - x_{n+1} = \delta_{n+1} = (A + LC)\delta_n -\mathbb{I}w_n - %L^{\top}z_{n}.
%\end{equation}
%So we can write
\begin{equation}
\delta_{n} = (A + LC)^{n}\delta_{0} - \sum_{k=0}^{n-1} (A+LC)^{n-1-k}(\mathbb{I} w_k + L^{\top}z_{k}).
\end{equation}
Assuming $\delta_0 = 0$, we have that
\begin{equation}
\delta_{n}= \sum_{k=0}^{n-1} \bar{D}_{k}w_k + \bar{L}_{k}z_k.
\end{equation}
Now, we can define the following:
\begin{equation}
\begin{aligned}
&C\delta_{n}\delta_n^{\top} C^{\top}=\\ &C\left(\sum_{k=0}^{n-1} \bar{D}_{k}w_k + \bar{L}_{k}z_k\right)\left(\sum_{k=0}^{n-1} \bar{D}_{k}w_k + \bar{L}_{k}z_k\right)^{\top} C^{\top},
\end{aligned}
\end{equation}
and obtain the expectation directly:
\begin{equation} \label{ExpEq}
\begin{aligned}
\mathbb{E}[(C\delta_{n}& - z_n)(C\delta_{n} - z_n)^{\top}]= \\
&C\left(\sum_{k=0}^{n-1} \bar{D}_{k}\Sigma_{w}\bar{D}^{\top}_{k} + \bar{L}_{k}\Sigma_{z}\bar{L}^{\top}_{k}\right) C^{\top} + \Sigma_{z}.
\end{aligned}
\end{equation}
Since $z_n$ and $\delta_n$ are independent for all $n$. Moreover, both system disturbances and measurement noise are independent. Under our \textit{key} assumption that both measurement noise and systems disturbances are given by i.i.d. bounded random vectors we have that
\begin{equation}
\|\delta_{n}\| \leq \sum_{k=0}^{n-1} \norm{\bar{D}_{k}} K_{w} + \norm{\bar{L}_{k}}K_{z},
\end{equation}
and that
\begin{multline}
\|C\delta_{n} - z_n\| \leq \\ K_z + \sum_{k=0}^{n-1} \norm{C\bar{D}_{k}} K_{w} + \norm{C\bar{L}_{k}}K_{z} = \bar{K}_{n}.
\end{multline}
So we have $(C\delta_{n} - z_{n})(C\delta_{n} - z_{n})^{\top} \preceq \bar{K}^{2}_{n}\mathbb{I}.$
\end{proof}

Now consider the matrix (\ref{eqn:acttestjoint}) that was used in the introduction to define the asymptotic tests. But now, instead of letting $n$ go to infinity, we keep it finite and then analyze the finite summation of matrices. Let $k' = \min\{k\geq 0\ |\ C(A+BK)^kB \neq 0\}$. The existence of such $k'$ is guaranteed (see \cite{hespanhol2017dynamic}). Moreover, define
\begin{equation}
\psi^{\top}_n =  \begin{bmatrix} (C\hat{x}_n - y_n)^{\top} & e^{\top}_{n-k'-1}\end{bmatrix}.
\end{equation}
Then we have: 
\begin{flalign}
& \frac{1}{N}\sum_{n=0}^{N-1} \psi_n \psi_n^{\top} = \nonumber&\\
&  \frac{1}{N} \scalebox{0.9}{$\begin{bmatrix} \sum_{n=0}^{N-1} (C\hat{x}_n - y_n)(C\hat{x}_n - y_n)^{\top} & \sum_{n=0}^{N-1} (C\hat{x}_n - y_n)e^{\top}_{n-k'-1} \\ \sum_{n=0}^{N-1} e_{n-k'-1}(C\hat{x}_n - y_n)^{\top} & \sum_{n=0}^{N-1} e_{n-k'-1}e^{\top}_{n-k'-1}\end{bmatrix}$}&
\end{flalign}
It suits our purposes to make sure that the above matrix is centered (that is have zero expected value). In order to achieve this, we construct the matrix
\begin{equation}
\begin{aligned}
 \frac{1}{N}\sum_{n=0}^{N-1} \Psi_n &= \frac{1}{N}\sum_{n=0}^{N-1} \psi_n \psi_n^{\top} - \\
 & \frac{1}{N} \scalebox{0.9}{$\begin{bmatrix} \sum_{n=0}^{N-1} \mathbb{E}[(C\delta_{n} - z_n)(C\delta_{n} - z_n)^{\top}] & 0 \\ 0 & N\Sigma_{e}\end{bmatrix}$}
\end{aligned}
\end{equation}
Note that it follows that: $\mathbb{E}[ \Psi_n] = 0 , \forall n \geq 0$, since $C\hat{x}_n - y_n = C\delta_{n} - z_n $. We wish to control the singular values of the above matrix. We will do so by analyzing each individual block. To ease the notation we define
\begin{equation}
 \Phi_{N} = \frac{1}{N}\sum_{n=0}^{N-1} \Psi_{n} 
\end{equation}
and we define each submatrix
\begin{align}
\Phi^{(1)}_{N} &= \frac{1}{N}\sum_{n=0}^{N-1} (C\hat{x}_n - y_n)(C\hat{x}_n - y_n)^{\top}-\nonumber\\ &\quad\quad\quad\frac{1}{N}\sum_{n=0}^{N-1}\mathbb{E}[(C\hat{x}_n - y_n)(C\hat{x}_n - y_n)^{\top}] \\
\Phi^{(2)}_{N} &= \frac{1}{N} \sum_{n=0}^{N-1} (C\hat{x}_n - y_n)e^{\top}_{n-k'-1}\\
\Phi^{(3)}_{N} &= \frac{1}{N} \sum_{n=0}^{N-1} (e_{n-k'-1}e^{\top}_{n-k'-1} - \Sigma_{e} )
\end{align}
such that
\begin{equation}
\label{eqn:defphin}
\Phi_{N} = \begin{bmatrix}
\Phi^{(1)}_{N} & \Phi^{(2)}_{N} \\ (\Phi^{(2)}_{N})^{\top} & \Phi^{(3)}_{N}
\end{bmatrix}.
\end{equation}
Our next step is to bound the norm of $\Phi^{(1)}_{N}$.
\begin{theorem}
\label{thm:phi1n}
If $v_n\equiv 0$ for all $n\geq 0$, then the following concentration inequality holds for all $N\geq 1$ and all $t$:
\begin{equation}
\mathrm{P} \bigg( \norm{\Phi^{(1)}_{N}} \geq t\Bigg) \leq m \cdot e^{-N^{2}t^2 / c^{(1)}_{N}}
\end{equation}
where $c^{(1)}_{N} = 8 \norm{\sum_{k=0}^{N-1} \left(\bar{K}^{4}_{k}\mathbb{I}\right) }$.
\end{theorem}

\begin{proof}
We start by defining the matrix $Y_n$ as
\begin{equation}
Y_n = (C\delta_n - z_n)(C\delta_n - z_n)^{\top} - \mathbb{E}[(C\delta_n - z_n)(C\delta_n - z_n)^{\top}].
\end{equation}
Now define a vector of independent i.i.d. Radamacher random variables $\{\epsilon_n\}_{n=0}^{N-1}$. We use the symmetrization property to write 
\begin{equation}
\norm{\Phi^{(1)}_{N}} \leq \norm{\frac{1}{N}\sum_{n=0}^{N-1} Y_n} \leq \norm{\frac{2}{N}\sum_{n=0}^{N-1} \epsilon_n Y_n}.
\end{equation}
Now we define a filtration $Z =(Y_n)_{n\geq 1}$ where $W_n = \epsilon_n Y_n , n \geq 1$. Then we see that each summand $W_n$ is conditionally independent given $Z$, because the Radamacher random variables are all i.i.d. This allows us to use the Hoeffding Bound for conditionally independent sums to obtain
\begin{multline}
\mathrm{P}\left( \norm{\frac{1}{N}\sum_{n=0}^{N-1} Y_n} \geq t \right) \leq\\
\mathrm{P} \left( \norm{\frac{2}{N}\sum_{n=0}^{N-1} \epsilon_nY_n} \geq t\right)%\mathrm{P} \left( \norm{\frac{2}{N}\sum_{n=0}^{N-1} \epsilon_nY_n} \geq t\right) 
\leq d \cdot e^{-N^{2}t^2 / 8\sigma^2}
\end{multline}
for $\sigma^2 = \norm{\sum_{k=0}^{N-1} \left(\bar{K}^{4}_{k}\mathbb{I}\right) }$.
The first inequality follows from applying the Laplace transform method and using the property
\begin{equation}
\mathbb{E}\left[\text{tr}\left( \textup{e}^{\sum_{i=1}^{n}X_{i}}\right)\right] \leq \mathbb{E}\left[\text{tr}\left( \textup{e}^{2\sum_{i=1}^{n}\epsilon_i X_{i}}\right)\right].
\end{equation}
We also used the fact $W^{2}_{k} \preceq \bar{K}^{4}_{k} \mathbb{I}$ for all $k$, and the fact that
\begin{equation}
\frac{1}{2}\norm{ \sum_{k} \bar{K}^{4}_{k} \mathbb{I} + E[W^2_{k}|(W_{j})_{j \neq k}]} \leq \norm{\sum_{k} \left(\bar{K}^{4}_{k}\mathbb{I}\right) }
\end{equation}
since $\mathbb{E}[W^2_{k}|(W_{j})_{j \neq k}] = \mathbb{E}[Y^2_{k}|(W_{j})_{j \neq k}] \preceq \bar{K}^{4}_{k} \mathbb{I}$.
\end{proof}

Next, we provide a bound on the norm of $\Phi^{(2)}_{N}$. But before that we need the following proposition:
\begin{proposition}
\label{prop:owh}
Let $e = (e_1,...,e_k,...,e_n)$ be a sequence of random vectors taking values in a space $\mathcal{Z}$. Now construct an exchangeable pair $ e' = (e_1,...,e'_k, ..., e_n)$ where $e_k$ and $e'_k$ are conditionally i.i.d. given $(e_j)_{j \neq k}$ and $k$ is an independent coordinate drawn uniformly from $\{1, ..., n\}$. We define 
\begin{equation}
H(e) = \begin{bmatrix} 0 & \sum_{n=0}^{N-1} (d_n)e^{\top}_{n-k'-1} \\ \sum_{n=0}^{N-1} e_{n-k'-1}(d_n)^{\top} & 0 \end{bmatrix}
\end{equation}
where $d_n =(C\delta_n - z_n)$. If $v_n\equiv 0$ for all $n\geq 0$, then the function $H(e)$ satisfies the bounded differences property
\begin{equation}
\begin{aligned}
 \mathbb{E}[(H(e) - H(e'))^2 | e] \preceq \bar{P}^{2}_{n}
\end{aligned}
\end{equation}
for $\bar{P}^{2}_{n}=  \max\{P^{2}_n,P^{'2}_n\}\mathbb{I}$ with positive constants $P^{2}_n,P^{'2}_n$:
\begin{equation}
P^{'2}_{n} = \norm{\mathbb{E}[Q'_n | e]} \leq \bar{K}^{2}_{n}(K^{2}_e + \norm{\Sigma_{E}}) 
\end{equation}
\begin{align}
P^{2}_n = \left(K^{2}_e + \textup{tr}(\Sigma_{E}\right)\times\\
\norm{C\left(\sum_{k=1}^{n-1} \bar{D}_{k}\Sigma_{w}\bar{D}^{\top}_{k} + \bar{L}_{k}\Sigma_{z}\bar{L}^{\top}_{k}\right) C^{\top} + \Sigma_{z}}
\end{align}
\end{proposition}

% \mattp{The way this sentence reads sounds like the reader should know what $P^{2}_n$ and $P^{'2}_n$ are but they arent defined till the proof.} \pedro{I added their definition to the proposition scope}

\begin{proof}
Let $q_n = d_n e^{\top}_{n-k'-1} - d_ne^{'\top}_{n-k'-1}$ and observe that
\begin{multline}
 \mathbb{E}[(H(e) - H(e'))^2 | e] =  \mathbb{E}\bigg[\begin{bmatrix} 0 & q_n \\ q_n^\top & 0 \end{bmatrix}^2 | e\bigg] = \\
 \mathbb{E}\bigg[\begin{bmatrix} Q_{n} & 0 \\ 0 & Q'_{n} \end{bmatrix} | e\bigg]
\end{multline}
where we have defined
\begin{equation}
\begin{aligned}
    Q_{n} = d_n e^{\top}_{n-k'-1} e_{n-k'-1}d^{\top}_{n} + d_ne^{'\top}_{n-k'-1} e^{'}_{n-k'-1}d^{\top}_{n}\\
    Q'_{n} = e_{n-k'-1}d^{\top}_n  d_{n}e^{\top}_{n-k'-1} + e^{'}_{n-k'-1}d^{\top}_n  d_{n}e^{'\top}_{n-k'-1}
\end{aligned}
\end{equation}
Now we have
\begin{multline}
\mathbb{E}[Q_n | e] =\\
\mathbb{E}[d_n e^{\top}_{n-k'-1} e_{n-k'-1}d^{\top}_{n} + d_ne^{'\top}_{n-k'-1} e^{'}_{n-k'-1}d^{\top}_{n} | e] =\\
(e^{\top}_{n-k'-1} e_{n-k'-1}) \mathbb{E}[d_{n}d^{\top}_n | e] + \\
\mathbb{E}[e^{'\top}_{n-k'-1} e^{'}_{n-k'-1} | e] \mathbb{E}[d_{n}d^{\top}_n | e]
\end{multline}
Recalling that $\norm{e_{k}} \leq K_{e}~\forall k \geq 0$ and (\ref{ExpEq}), it follows that
\begin{multline}
\norm{\mathbb{E}[Q_n | e]} \leq \left(K^{2}_e + \text{tr}(\Sigma_{E}\right)\times\\
\norm{C\left(\sum_{k=1}^{n-1} \bar{D}_{k}\Sigma_{w}\bar{D}^{\top}_{k} + \bar{L}_{k}\Sigma_{z}\bar{L}^{\top}_{k}\right) C^{\top} + \Sigma_{z}}= P^{2}_n.
\end{multline}
Moreover, it follows that
\begin{multline}
\norm{\mathbb{E}[Q'_n | e]} = \\
e_{n-k'-1}d^{\top}_n  d_{n}e^{\top}_{n-k'-1} + e^{'}_{n-k'-1}d^{\top}_n  d_{n}e^{'\top}_{n-k'-1}\\
=(\mathbb{E}[(d^{\top}_{n} d_{n}) | e]) e_{n-k'-1}e^{\top}_{n-k'-1} + \\ 
 (\mathbb{E}[d^{\top}_{n} d_{n}| e]) \mathbb{E}[ e^{'}_{n-k'-1}e^{'\top}_{n-k'-1} | e]
\end{multline}
So we get
\begin{equation}
\norm{\mathbb{E}[Q'_n | e]} \leq \bar{K}^{2}_{n}(K^{2}_e + \norm{\Sigma_{E}}) = P^{'2}_{n}
\end{equation}
Hence it follows that
\begin{equation}
\begin{aligned}
\norm{\mathbb{E}[(H(e) - H(e'))^2 | e]} \leq \max\{P^{2}_n,P^{'2}_n\}
\end{aligned}
\end{equation}
So it follows that
\begin{equation}
\begin{aligned}
 \mathbb{E}[(H(e) - H(e'))^2 | e] \preceq \bar{P}^{2}_{n}
\end{aligned}
\end{equation}
where $\bar{P}^{2}_{n} = \max\{P^{2}_n,P^{'2}_n\}\mathbb{I}$.
\end{proof}
Now we are ready to provide our theorem.
\begin{theorem}
\label{thm:phi2n}
If $v_n\equiv 0$ for all $n\geq 0$, then the following concentration inequality holds for all $N\geq 1$ and all $t$:
\begin{equation}
\mathrm{P} \bigg( \norm{\Phi^{(2)}_{N}} \geq t\Bigg) \leq (m+p) \cdot e^{-N^{2}t^2 / c^{(2)}_{N}}
\end{equation}
where $c^{(2)}_{N} = \norm{\sum_{k=0}^{N-1} \bar{P}^{2}_{k} }$ for $\bar{P}^{2}_{k} = \max\{P^{2}_k,P^{'2}_k\}\mathbb{I}$, where
\begin{equation}
\begin{aligned}
P^{2}_{k} &= \left(K^{2}_e + \textup{tr}(\Sigma_{E}\right)\times \\
&\quad\quad\norm{C\left(\sum_{k=1}^{n-1} \bar{D}_{k}\Sigma_{w}\bar{D}^{\top}_{k} + \bar{L}_{k}\Sigma_{z}\bar{L}^{\top}_{k}\right) C^{\top} + \Sigma_{z}} \\
P'^{2}_{k} &= \bar{K}^{2}_{n}(K^{2}_e + \norm{\Sigma_{E}}).
\end{aligned}
\end{equation}
\end{theorem}

\begin{proof}
We wish to provide bounds on the operator norm of
\begin{equation}
 \Phi^{(2)}_{N} =\frac{1}{N}\sum_{n=0}^{N-1} (C\delta_n - z_n)e^{\top}_{n-k'-1}
\end{equation}
To achieve that, we will use the concept of matrix Stein pairs as defined previously. Let $E = (e_1,...,e_k,...,e_n)$ be a sequence of random vectors taking values in a space $\mathcal{Z}$. Now construct an exchangeable pair $E' = (e_1,...,e'_k, ..., e_n)$ where $e_k$ and $e'_k$ are conditionally i.i.d. given $(e_j)_{j \neq k} $ and $k$ is an independent coordinate drawn uniformly from $\{1, ..., n\}$. We define $H(e)$ as in Proposition \ref{prop:owh}:
\begin{equation}
H(e) = \begin{bmatrix} 0 & \sum_{n=0}^{N-1} b_ne^{\top}_{n-k'-1} \\ \sum_{n=0}^{N-1} e_{n-k'-1}b_n^{\top} & 0 \end{bmatrix}
\end{equation}
where $b_n = C\delta_n - z_n$. Since $\mathbb{E}(H(e)) = 0$, this means $H(e)$ satisfies the self-reproducing property
\begin{equation}
\sum_{n=1}^{N} H(e) - \mathbb{E}[H(e) | (e_j)_{j\neq (n-k'-1)} ] = H(e)
\end{equation}
for the choice of parameter $s = 1$ (see (\ref{eqn:onewiths}) for the definition of $s$), since for all $n \in \{1,...,N\}$ we have
\begin{multline}
H(e) - \mathbb{E}[H(e) | (e_j)_{j\neq (n-k'-1)} ] =\\
\begin{bmatrix} 0 & (C\delta_n - z_n)e^{\top}_{n-k'-1} \\ e_{n-k'-1}(C\delta_n - z_n)^{\top} & 0 \end{bmatrix}
\end{multline}
Next, we use Proposition \ref{prop:owh} to state that $H(e)$ also satisfies the bounded differences property. So we have
\begin{equation}
 \mathbb{E}[(H(e) - H(e')^2 | e] \preceq \bar{P}^{2}_{n}
\end{equation}
for $\bar{P}^{2}_{n} = \max\{P^{2}_n,P^{'2}_n\}\mathbb{I}$. Hence, we apply the McDiarmid inequality to the dilation $H(e) \in \mathcal{H}^{m+p}$ to obtain
\begin{multline}
 \mathrm{P}\Bigg(\norm{\frac{1}{N}H(e)} \geq t\Bigg) = \\
 \mathrm{P}\Bigg(\norm{\frac{1}{N}\sum_{n=0}^{N-1} (C\delta_n - z_n)e^{\top}_{n-k'-1}} \geq t\Bigg) \leq\\
 (m+p) \cdot e^{-N^{2}t^2 / L}
\end{multline}
for $L = \norm{ \sum_{k=0}^{N-1} \bar{P}^{2}_{k} }$.
\end{proof}
Now we focus on bounding the last submatrix  $\Phi^{(3)}_{N}$.
\begin{theorem}
The following concentration inequality holds for all $N\geq 1$ and all $t$:
\begin{equation}
\mathrm{P} \bigg( \norm{\Phi^{(3)}_{N}} \geq t\Bigg) \leq 2q \cdot e^{-N^{2}t^2 / c^{(3)}_{N}}
\end{equation}
where $c^{(3)}_{N} = \norm{\sum_{k=0}^{N-1} (\bar{K}^{2}_{e}\mathbb{I} - \Sigma_{e})^{2} + \mathbb{E}[(e_n e^{\top}_{n})^{4}] - \Sigma^{2}_{e}}$.
\end{theorem}

\begin{proof}
We wish to provide a bound on the norm of
\begin{equation}
\Phi^{(3)}_{N} =  \frac{1}{N} \sum_{n=0}^{N-1} (e_{n-k'-1}e^{\top}_{n-k'-1} - \Sigma_{e})
\end{equation}
Define $\bar{E}_n = e_{n-k'-1}e^{\top}_{n-k'-1} - \Sigma_{e}$. We apply the Hoeffding bound for the independent sum to obtain
\begin{equation}
\mathrm{P} \left( \norm{\frac{1}{N}\sum_{n=0}^{N-1} \bar{E}_n}  \geq t \right) \leq d \cdot e^{-N^{2}t^2 / 2\sigma^2}
\end{equation}
for $\sigma^2 = \frac{1}{2}\norm{\sum_{k=0}^{N-1} \left(\bar{K}^{2}_{e}\mathbb{I} - \Sigma_{e}\right)^{2} + \mathbb{E}[(e_n e^{\top}_{n})^{4}] - \Sigma^{2}_{e}}$, since
\begin{equation}
\bar{E}_n^{2} \preceq \sum_{k=0}^{N-1} \left(\bar{K}^{2}_{e}\mathbb{I} - \Sigma_{e}\right)^{2}
\end{equation}
and by the definition of expectation we have that $\mathbb{E}\left[\bar{E}_n^{2}\right] = \mathbb{E}[(e_{n-k'-1} e^{\top}_{n-k'-1})^{4}] - \Sigma^{2}_{e}$.
\end{proof}

\section{Finite Sample Tests for General LTI Systems}

\label{Sec5}

In this section, we provide our finite sample tests based on dynamic watermarking for general LTI Systems with switching. In the previous section, we obtained concentration inequalities for each of the submatrices of $\Phi_N$ (\ref{eqn:defphin}). Note $\Phi_3$ is the private excitation matrix we get to design, and so it is in our power to choose the dynamic watermark to display a desired concentration behavior.

We are now ready to state the main theorem of this work, which basically characterizes the behavior of a switching rule based on the finite-time concentration inequalities. Our switching rule is constructed by thresholding the block submatrices $\Phi^{(1)}_N$ and $\Phi^{(2)}_N$ using the measurements of the second observer ((\ref{xhat}) and (\ref{deltahat})) and applying the switch on the first observer once those thresholds are violated, and the we switch back on violations disappear. Let $S$ be a positive constant such that $\max\{c^{(1)}_{N},c^{(2)}_{N},c^{(3)}_{N}\} \leq NS$; such an $S$ exists when $(A+BK)$ and $(A+L_nC_n)$ are Schur stable provided that the switching rule satisfies the condition specified in Proposition \ref{prop:swr}.
\begin{theorem}
\label{thm:main}
Recall the closed-loop MIMO LTI system (\ref{LTI-SYS}) with $\alpha_{n}$ being our switching control action that chooses between two different observation matrices. Define the threshold $t_{N} = \sqrt{(1+\rho)S\log N / N}$, where $\rho > 0$. Let $\Phi^{(1)}_{N}$ and $\Phi^{(2)}_{N}$ be defined using the measurements from Eq. \ref{xhat} and Eq. \ref{deltahat}. Let $\alpha_{N}$ be the switching decision rule with
\begin{itemize}
\item we choose the switching input $\alpha_N = 0$ when we have $\norm{\Phi^{(1)}_{N}} <  t_{N}$ or $\norm{\Phi^{(2)}_{N}} < t_{N}$
\item we switch from $\alpha_{N-1} = 0$ to $\alpha_N = 1$ when $\alpha_{N-i} = 0$ for $i\in\{1,\ldots,\tau\}$ and $\norm{\Phi^{(1)}_{N}} \geq  t_{N}$ and $\norm{\Phi^{(2)}_{N}} \geq t_{N}$.
\end{itemize}
Moreover, let $E_{N}$ for all $n \geq 1$ denote the event
\begin{equation}
%E_{n} = \Bigg[ \norm{\Phi_{n}} \geq t_{n} + (t_{n})^{2} + t_{n} \Bigg].
E_{N} = \Bigg[ \norm{\Phi^{(1)}_{N}} >  t_{N} \bigcup \norm{\Phi^{(2)}_{N}} > t_{N}\Bigg]
\end{equation}
Then if $v_N\equiv 0$ for all $N\geq 0$, we have that
\begin{equation}
\mathrm{P}(\limsup_{N \rightarrow \infty}E_{N} ) = 0.
\end{equation} 
That is, under no attacks our switching rule incorrectly switches the system only a finite number of times.
\end{theorem}

\begin{proof}
Recall that we previously proved the following matrix concentration inequalities for each submatrix:
\begin{align}
\mathrm{P} \bigg( \norm{\Phi^{(1)}_{N}} \geq t_{N}\Bigg) &\leq m \cdot e^{-N^{2}t_N^2 / c^{(1)}_{N}}\\
\mathrm{P}\Bigg(\norm{\Phi^{(2)}_{N}} \geq t_{N}\Bigg) &\leq (m+p) \cdot e^{-N^{2}t_N^2 / c^{(2)}_{N}}
%\mathrm{P} \bigg( \norm{\Phi^{(3)}_{N}}  \geq t_{N} \bigg) &\leq m \cdot e^{-N^{2}t_N^2 / c^{(3)}_{N}}
\end{align}
for the constants $c^{(1)}_{N}$ and $c^{(2)}_{N}$. Summing over all $N$, we have
\begin{equation}
\sum_{k=1}^{\infty} \mathrm{P} \bigg( \norm{\Phi^{(j)}_{k}} \geq t_{k}\Bigg) \leq (m+p) \int_1^\infty \frac{1}{k^{1+\rho}}dk < \infty.
\end{equation}
Hence the Borel-Cantelli Lemma implies that for the event
\begin{equation}
E^{(j)}_{N} = \Bigg[ \norm{\Phi^{(j)}_{N}} \geq t_{N} \Bigg]% , ~N \geq 1,
\end{equation} 
we have
\begin{equation}
\mathrm{P}(\limsup_{N \rightarrow \infty} E^{(j)}_{N} ) = 0 , ~\forall j = \{1,2,3\}.
\end{equation} 
Now, if we define the event
\begin{equation}
%E_{N} = \left[ \norm{\Phi_{N}} \geq t^{(1)}_{N} + (t^{(2)}_{N})^{2} + t^{(3)}_{N} \right] , ~N \geq 1,
E_{N} = \Bigg[ \norm{\Phi^{(1)}_{n}} >  t_{n} \bigcup \norm{\Phi^{(2)}_{n}} > t_{n}\Bigg] , ~N \geq 1,
\end{equation} 
%we have the following inclusions of events
%\begin{equation}
%E_{N} \subset \bigcup_{j=1}^{3} E^{(j)}_{N},
%\end{equation}
then it follows that
\begin{equation}
\mathrm{P}(E_{N}) \leq \mathrm{P}\left(\bigcup_{j=1}^{2} E^{(j)}_{N}\right) \leq \sum_{j=1}^{2} \mathrm{P}\left(E^{(j)}_{N}\right) , ~N \geq 1.
\end{equation}
So summing once more for all $N$ gives
\begin{multline}
\sum_{k=1}^{\infty} \mathrm{P}(E_{k}) \leq \sum_{k=1}^{\infty}\sum_{j=1}^{2} \mathrm{P}\left(E^{(j)}_{N}\right) < \\
2(m+p) \int_1^\infty \frac{1}{k^{1+\rho}}dk < \infty.
\end{multline}
We obtain by applying Borel-Cantelli lemma that
\begin{equation}
\mathrm{P}\left(\limsup_{N \rightarrow \infty}E_{N} \right) = 0,
\end{equation} 
which is the desired result.
\end{proof}

The result of this theorem implies that if there is no attack to the system, the operator norm of the matrices involved can have ``large'' deviations only a finite number of times, hence we obtain that a switching rule based on tests derived from the concentration inequalities defined previously will only trigger attack alerts only a finite number of times. In addition, we note that the having a second observer to compute the finite tests is the key to ensure that the concentration inequalities are consistent with the obtained measurements. While the first observer measurements with switching plays the role in the control synthesis. Lastly, we observe that we do not need to enforce the test on $\norm{\Phi^{(3)}_{3}}$ since this submatrix is only composed of the watermaking signal, and the attacks do not have the power to affect the watermarking imposed by the controller.

\section{Attack Magnitude Thresholding}

\label{Sec6}

The previous section gives a finite sample test that works properly when there is no attack. 
Our goal here is to determine the trade-off between our test's statistical power and the attack magnitude. 
Namely, we are interested in how the right-hand side of our finite sample tests are related to the magnitude of the attack vectors. To do so, we consider the first observation matrix under an attack $y_{n} = Cx_{n} + z_n + v_n$, where $v_n$ is an additive attack. We first consider attacks that are small perturbations and then consider more complex attacks of the form explored in \cite{hespanhol2017dynamic} . 
Note we have again omitted the subscript of the observation matrix for clarity. Namely, in the next two subsections we will consider two attack forms. 

\subsection{Perturbation Attacks}

The first attack we analyze is when $v_n$ consists of a small perturbation that could be determinstic and/or stochastic. To begin our analysis, let $\bar{\delta}_n$ be the measurement error when the system is under attack, and observe that
\begin{equation}
\bar{\delta}_{n+1} = (A+LC)\bar{\delta}_{n} - Dw_n -L^{\top}z_{n} -L^{\top}v_{n}. 
\end{equation}
Expanding this expression gives that
\begin{equation}
\bar{\delta}_{n} = (A + LC)^{n}\bar{\delta}_{0} - \sum_{k=0}^{n-1} (A+LC)^{n-1-k}(\mathbb{I} w_k - L^{\top}z_{k} - L^{\top}v_{k})
\end{equation}
where $\bar{\delta}_{0} = \delta_{0} = 0$. So we can rewrite the above as
\begin{equation}
\bar{\delta}_{n} = \sum_{k=0}^{n-1} \bar{D}_{k}w_k + \bar{L}_{k}z_k  + \bar{L}_{k}v_k = \delta_{n} + \sum_{k=0}^{n-1} \bar{L}_{k}v_k.
\end{equation}
Next we define $V_{n} = C\sum_{k=0}^{n-1} \bar{L}_{k}v_k - v_n$, and observe that the quantity $V_n$ is determined by the attacker since it depends upon the values of $v_k$. Qualitatively, we note that the magnitude of $V_n$ is related to the attack magnitude, since if there is no attack then $V_n\equiv 0$ for all $n$.
% \mattp{This makes it sound like a stronger magnitude of attack will always lead to a larger $V_n$ but that isnt true for systems that are not open loop stable.}\pedro{I rephrased it a bit. But perhaps it is better to cite or explain why this is only true for open-loop stable systems?}

\begin{theorem}
\label{thm:fam}
Consider the closed-loop MIMO LTI system (\ref{LTI-SYS}) with $\alpha_{n},t_{N},E_n$ as defined in Theorem \ref{thm:main}, and suppose the attacker chooses the perturbation attack described above. If the attack values $v_k$ are such that there exists a positive constant $G$ with
\begin{equation}
\label{eqn:ram}
%\frac{1}{N}\sum_{k=0}^{N-1} \norm{V_{N}} \leq \frac{\sqrt{\min\{c^{(1)}_N,c^{(2)}_N}\}}{2K_eN}.
\frac{1}{N}\sum_{k=0}^{N-1} \norm{V_{k}} \leq \frac{G}{N}.
\end{equation}
then we have that $\mathrm{P}(\limsup_{n \rightarrow \infty}E_{n} ) = 0$. That is, under a perturbation attack with the above specifications the attack is detected only a finite number of times.
\end{theorem}

\begin{proof}
We begin by considering
\begin{equation}
\begin{aligned}
\Phi^{(1)}_N = \frac{1}{N} \sum_{n=0}^{N-1}(C\hat{x}_n - y_n)(C\hat{x}_n - y_n)^{\top} - \\
\frac{1}{N}\sum_{n=0}^{N-1}\mathbb{E}[(C\hat{x}_n - y_n)(C\hat{x}_n - y_n)^{\top}] = \\
 \frac{1}{N}  \sum_{n=0}^{N-1}(C\bar{\delta}_{n} - z_n - v_n)(C\bar{\delta}_{n} - z_n - v_n)^{\top} - \\
 \frac{1}{N}\sum_{n=0}^{N-1}\mathbb{E}[(C\hat{x}_n - y_n)(C\hat{x}_n - y_n)^{\top}] = \\
 \frac{1}{N}  \sum_{n=0}^{N-1}(C\delta_n - z_n)(C\delta_n - z_n)^{\top} + D_{n} + D^{\top}_{n} + M_{n} - \\
 \frac{1}{N}\sum_{n=0}^{N-1}\mathbb{E}[(C\hat{x}_n - y_n)(C\hat{x}_n - y_n)^{\top}].
\end{aligned}
\end{equation}
where $\delta_n$ is the measurement error under no attack, and 
\begin{equation}
\begin{aligned}
D_n &= \frac{1}{N}  \sum_{n=0}^{N-1}(C\delta_n - z_n)V_n^{\top}\\
M_n &= \frac{1}{N}\sum_{n=0}^{N-1}V_n V_n^{\top}.
\end{aligned}
\end{equation}
Now using Theorem \ref{thm:phi1n}, we have that
\begin{multline}
\mathrm{P}(\norm{\Phi^{(1)}_{N}} \geq t_N) \leq \\
\mathrm{P}\bigg(\|\frac{1}{N}\sum_{n=0}^{N-1}(C\delta_n - z_n)(C\delta_n - z_n)^{\top} - \\
\frac{1}{N}\sum_{n=0}^{N-1}\mathbb{E}[(C\hat{x}_n - y_n)(C\hat{x}_n - y_n)^{\top}] \| \geq t_N + \\
-2\norm{{D}_{N}} - \norm{M_{N}}\bigg) \leq\\
  m\text{ } e^{\frac{-N^2\left(t_{N} - 2\norm{{D}_{N}} - \norm{M_{N}} \right)^2}{c^{(1)}_{N}}}.
\end{multline}
Next observe that
\begin{multline}
 2\norm{\bar{D}_{N}} + \norm{M_{N}} \leq \frac{2\bar{K}_N}{N}\sum_{n=0}^{N-1} \norm{V_{n}} + \frac{1}{N}\sum_{n=0}^{N-1} \norm{V_{n}}^2 \leq \\
\frac{2\bar{K}_NG}{N} + \frac{G^2}{N}.
\end{multline}
Since $(A+LC)$ is Schur stable, then from the definition of $\bar{K}_N$ we immediately get that there exists a positive constant $\bar{S}$ such that $\bar{K}_N \leq \bar{S}$ for all $N\geq 1$.  Combining this with the above implies that
\begin{equation}
\sum_{k=1}^{\infty} \mathrm{P}(\norm{\Phi^{(1)}_{k}} \geq t^{(1)}_k) < \infty,
\end{equation}
and so the Borel-Cantelli lemma implies that $\norm{\Phi^{(1)}_{N}} \geq t^{(1)}_N$ only finitely many times. Our next step considers
\begin{multline}
\Phi^{(2)}_N = \frac{1}{N} \sum_{n=0}^{N-1}(C\hat{x}_n - y_n)e^{\top}_{n-k'-1} = \\
\frac{1}{N}  \sum_{n=0}^{N-1}(C\bar{\delta}_{n} - z_n - v_n)e^{\top}_{n-k'-1} =\\
\frac{1}{N} \sum_{n=0}^{N-1} (C\delta_n -z_n)e^{\top}_{n-k'-1} + H_n% = \\
%\Phi^{(2)}_{N} + H_{N}
\end{multline}
where 
\begin{equation}
H_{N} = \frac{1}{N} \sum_{n=0}^{N-1} V_ne^{\top}_{n-k'-1}.
\end{equation}
\begin{figure*}[t]
    \centering
    \includegraphics[trim={1cm 0.3cm 1cm 0},clip,width=0.8\textwidth]{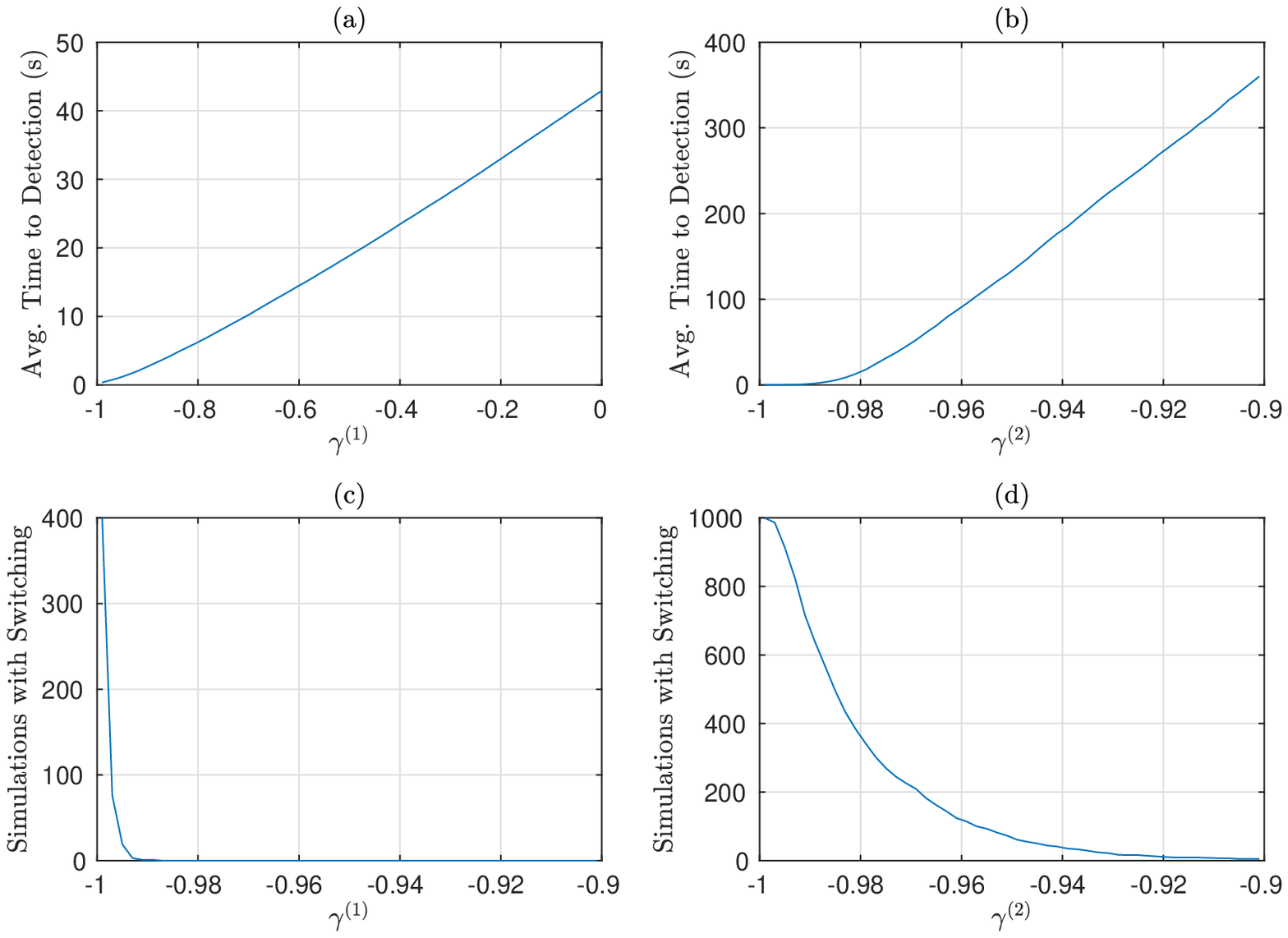}
    \caption{Caption}
    \label{fig:time2detect}
\end{figure*}
Now using Theorem \ref{thm:phi2n}, we have that
\begin{multline}
\mathrm{P}(\norm{\Phi^{(2)}_{N}} \geq t_N) \leq \\
\mathrm{P}\bigg(\|\frac{1}{N} \sum_{n=0}^{N-1} (C\delta_n -z_n)e^{\top}_{n-k'-1}\| \geq t_N - \norm{H_{N}}\bigg) \leq\\
  2m \text{ } e^{\frac{-N^2\left(t_{N} - \norm{H_{N}} \right)^2}{c^{(2)}_{N}}}.
\end{multline}
Next observe that
\begin{equation}
\norm{H_{N}} \leq \frac{K_e}{N}\sum_{k=0}^{N-1} \norm{V_{N}} \leq \frac{K_eG}{N}.
\end{equation}
Combining this with the above implies that
\begin{equation}
\sum_{k=1}^{\infty} \mathrm{P}(\norm{\Phi^{(2)}_{k}} \geq t_k) < \infty,
\end{equation}
and so the Borel-Cantelli lemma implies that $\norm{\Phi^{(2)}_{N}} \geq t_N$ only finitely many times. The remainder of the proof follows similarly to that of the last steps of Theorem \ref{thm:main}.
\end{proof}

Our analysis in this subsection is capable of only providing a simple relation between the power of our detection scheme and the magnitude of $V_{n}$. An analysis that translates to the bounds of each individual $v_{n}$ is more involved because it depends explicitly on the structure/behavior of the matrix $(A+LC)$.

\begin{figure*}[t]
    \centering
    \includegraphics[trim={2cm 0 2cm 0},clip,width=\textwidth]{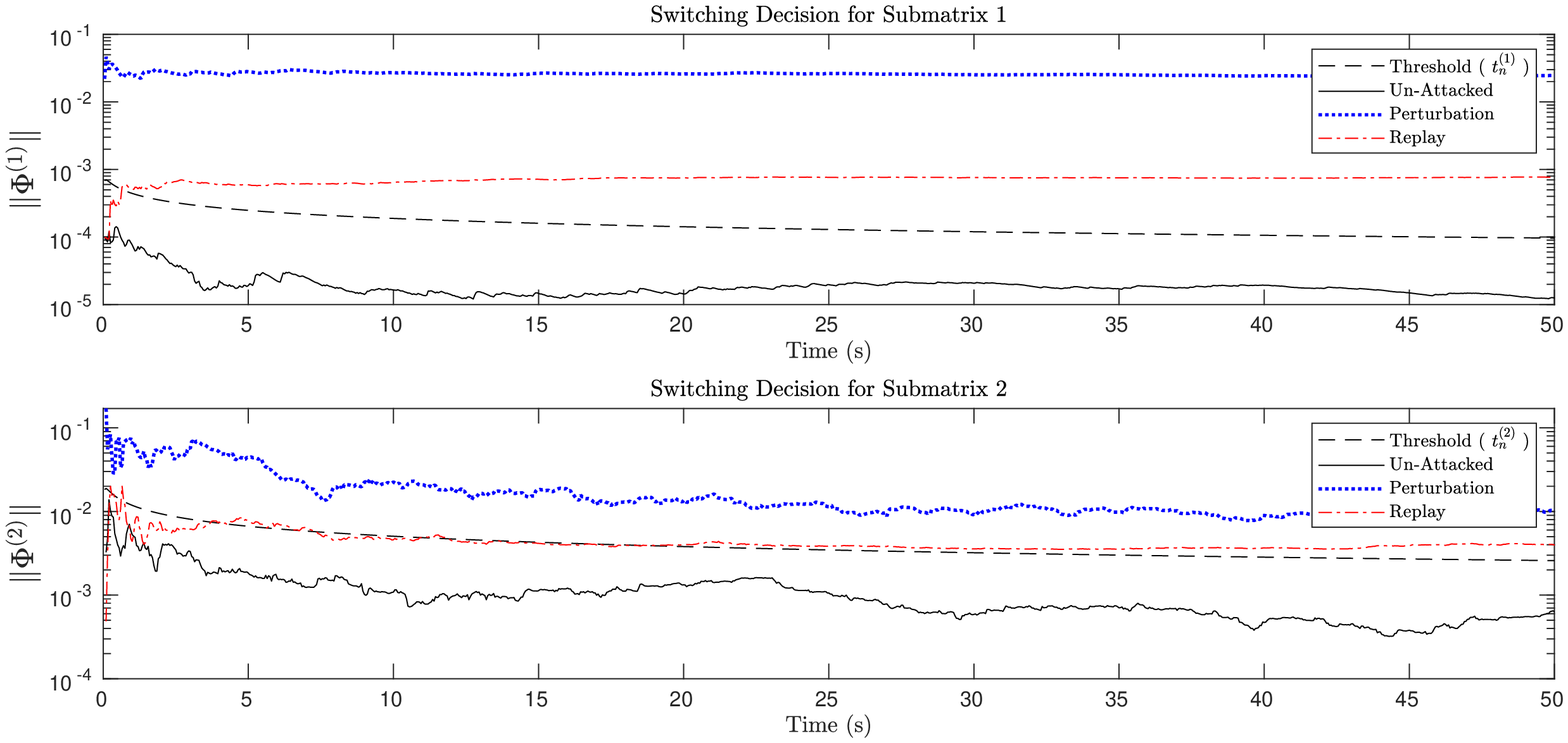}
    \caption{Switching Decision Values Based on the Submatrices $\|\Phi_n^{(1)}\|$ and $\|\Phi_n^{(2)}\|$ in Simulation of Autonomous Vehicle}  
    \label{fig:Test_values}
\end{figure*}

\begin{figure*}[t]
    \centering
    \includegraphics[trim={2cm 0 2cm 0},clip,width=\textwidth]{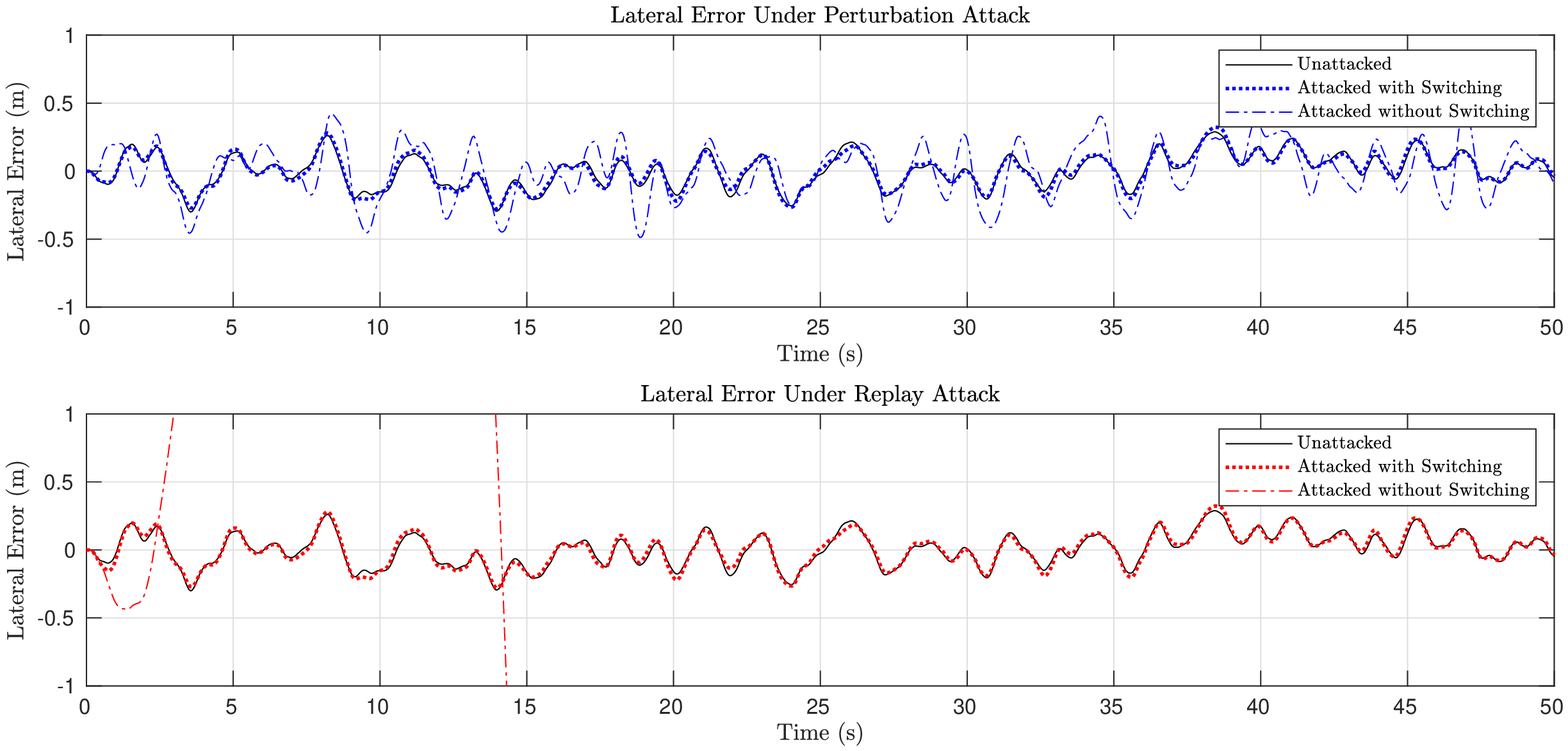}
    \caption{Performance Comparison of Simulated Autonomous Vehicle Lane Keeping Under Attack with and without Switching Policy}
    \label{fig:Performace_compare}
\end{figure*}

\subsection{Replay Attacks}

The second attack we analyze is when
\begin{equation}
v_n = C\xi_n + \zeta_n - (Cx_n + z_n)\label{eq:attack_form_2}
\end{equation}
where $\xi_{n+1} = (A+BK)\xi_n + \omega_n$ and $\omega_n$ is a bounded disturbance. This is a \emph{replay attack} \cite{langner2011stuxnet}, since it subtracts the real sensor measurements and substitutes these with a replay of the dynamics starting from a different initial condition. In fact, we will perform our analysis for a more general attack
\begin{equation}
v_n = C\xi_n + \zeta_n - \gamma\cdot(Cx_n + z_n)\label{eq:attack_form_3},
\end{equation}
where $\gamma\in\mathbb{R}$. This attack also allows for dampening or amplifying the true sensor measurements $(Cx_n + z_n)$.

\begin{theorem}
Consider the closed-loop MIMO LTI system (\ref{LTI-SYS}) with $\alpha_{n},t_{N},E_n$ as defined in Theorem \ref{thm:main}, and suppose the attacker chooses the attack (\ref{eq:attack_form_3}). If the attack is not trivial (i.e., a trivial attack has $v_n \equiv 0$ for all $n\geq 0$), then we have that $\mathrm{P}(\limsup_{n \rightarrow \infty}\neg E_{n} ) = 0$. That is, under the attack with the above specifications the attack is \emph{not} detected only a finite number of times.
\end{theorem}

\begin{proof}
Suppose $\gamma \neq 0$. Then the proof of Theorem 1 in \cite{hespanhol2017dynamic} shows that $\lim_{N\rightarrow\infty} \Phi^{(2)}_N$ exists almost surely and is not equal to 0. This means that $\mathrm{P}(\limsup_{n \rightarrow \infty}\neg E_N^{(2)} ) = 0$. Now consider the case $\gamma = 0$. Then the proof of Theorem 1 in \cite{hespanhol2017dynamic} shows that $\lim_{N\rightarrow\infty} \Phi^{(1)}_N$ exists almost surely and is not equal to 0. This means that $\mathrm{P}(\limsup_{n \rightarrow \infty}\neg E_N^{(1)} ) = 0$. The remainder of the proof by repeating the last steps of Theorem \ref{thm:main} for the two cases, after noting that $\neg E_N = \neg E_N^{(1)} \vee  \neg E_N^{(2)}$ by De Morgan's laws.
\end{proof}
This result is stronger than Theorem \ref{thm:fam} in that it says all replay attacks, and more generally attacks of the form (\ref{eq:attack_form_3}), will \emph{not} be detected by the finite sample tests only a finite number of times. In fact, this result is analagous to the zero-average-power results (\ref{eqn:zasp}) of past work on dynamic watermarking for LTI systems with general structure \cite{hespanhol2017dynamic}, since this result says that only (trivial) replay attacks with zero-average-power cannot be detected.

\section{Experimental Results}
\label{Sec7}
To further demonstrate the effectiveness of this method, we return to the lane keeping example used in \cite{hespanhol2017dynamic} which is based off of the standard model for lane keeping and speed control \cite{turri2013linear}. 
In this model the state vector takes the form $x^T=[\psi~y~s~\gamma~v]$ and input vector $u^T=[r~a]$, where $\psi$ is heading error, $y$ is lateral error, $s$ is trajectory distance, $\gamma$ is vehicle angle, $v$ is vehicle velocity, $r$ is steering, and $a$ is acceleration. 
Linearizing about a straight trajectory at a velocity of 10 m/s and step size of 0.05 seconds gives us an LTI system:
\begin{equation}
A=
\begin{bmatrix}
1 & 0 & 0 & \frac{1}{10} & 0\\
\frac{1}{2} & 1 & 0 & \frac{1}{40} & 0\\
0 & 0 & 1 & 0 & \frac{1}{2}\\
0 & 0 & 0 & 1 & 0\\
0 & 0 & 0 & 0 & 1
\end{bmatrix}
B=
\begin{bmatrix}
\frac{1}{400} & 0\\
\frac{1}{2400} & 0\\
0 & \frac{1}{800}\\
\frac{1}{20} & 0\\
0 & \frac{1}{20}
\end{bmatrix}
\end{equation}
with $C_1=C_2=[I,0]\in \mathbb{R}^{3\times 5}$. The process noise and watermark take the form of uniform random variables such that $w\in[-2.5\times10^{-4},~2.5\times10^{-4}]^5$ and $e\in[-2,~2]^2$. Similarly the measurement noise for each sensor is also estimated as uniform random variables where  $\zeta\in[-1\times10^{-2},~1\times10^{-2}]^3$ and $\eta\in[-2\times10^{-2},~2\times10^{-2}]^3$. For this example we can think of the $\zeta$ measurements as localization using visual or lidar based localization with high definition mapping, and $\eta$ as GPS localization. Finally controller and observer gains $K$ and $L_1=L_2$ were chosen to stabilize the closed loop system.

For this system it was found that
\begin{align}
    c_n^{(1)} &\leq 6.7502\times 10^{-5}n\label{eq:c1_bound}\\
    c_n^{(2)} &\leq 0.0968n.\label{eq:c2_bound}
\end{align}
Using the threshold structure defined in Theorem \ref{thm:main} results in
\begin{align}
    \tau_n^{(1)} &= \sqrt{(1+\rho^{(1)})(6.7502\times 10^{-5})\log(n)/n}\\
    \tau_n^{(2)} &= \sqrt{(1+\rho^{(2)})(0.968)\log(n)/n}.
\end{align}
While the finite switching guarantee given by Theorem \ref{thm:main} only applies for $\rho^{(1)},\rho^{(2)}>0$, due to the conservative nature of the bounds in \eqref{eq:c1_bound}-\eqref{eq:c2_bound} in addition to the desire to also maintain a sufficiently quick detection we instead heuristically tune these values to find the desired balance.

% For this system it was found that using exponential functions for the test parameters 
% \begin{align}
%     t^{'}_n&=5\times10^{-5}+5\times10^{-4}e^{-0.01n}\\
%     \overline{t}_n&=2\times10^{-3}+0.07e^{-0.01n}
% \end{align}
% resulted in an acceptable balance between false alarms and detector robustness. For these parameters, out of 1000 simulations of 1000 discrete time steps, 2 simulations resulted in switching. Furthermore, since $t^{'}_n<\sqrt{c_n^{(1)}/n}$ and $\overline{t}_n<\sqrt{c_n^{(2)}/n}$ for only a finite number of steps these parameters maintain the guarantee of a finite number of switches under no attacks. 

For our analysis of this system, we once again consider the two forms of attack discussed in Section \ref{Sec6}.
The perturbation attack takes the form of random noise pulled from a uniform distribution such that $v_n\in[-0.15 0.15]^3$
The replay attack is described in \eqref{eq:attack_form_2} where $\xi_0=0$ and $\zeta$ and $\omega$ are uniformly distributed such that $\zeta\in[-2.5\times10^{-4},~2.5\times10^{-4}]^3$ and $\omega\in[-2.5\times10^{-4},~2.5\times10^{-4}]^5$. 

Each attacked system, along with an un-attacked system were simulated 1000 times for 10,000 discrete time steps.
While the perturbation attack is detected and switching occurs almost immediately for $\rho^{(1)},\rho^{(2)}<1$, the replay attack can take a much longer time to be detected. 
Figure \ref{fig:time2detect} shows the average time to detection for each of our switching conditions in addition to the number of trials that result in switching for the un-attacked case plotted against the corresponding value of $\rho^{(1)}$ or $\rho^{(2)}$.
While the number of switching simulations for the un-attacked system under switching condition 2 appear to be quite large even when the average time to detect is relatively large, it is important to note that many of the unwanted switches occur in the first four discrete steps which can be mitigated in practice by ignoring the first four values.

Choosing values of $\rho^{(1)}=\rho^{(2)}=-0.98$, each attack was again simulated this time for 1000 discrete time steps both with and without the switching policy. 
Figure \ref{fig:Test_values} shows the value of $\|\Phi_n^{(1)}\|$ and $\|\Phi_n^{(2)}\|$ for both normal operation and under each of the attacks when the switching policy is not being used. The plot shows that for both attack 1 and attack 2 the switching policy will result in an almost immediate and consistent transfer from the attacked sensor to the protected sensor. Furthermore, when the system is un-attacked the values of $\|\Phi_n^{(1)}\|$ and $\|\Phi_n^{(2)}\|$ remain below the switching threshold.
Figure \ref{fig:Performace_compare} compares the performance of the lane keeping algorithm for each attack with respect to the un-attacked performance both with and without the switching policy. This plot shows that for both attacks the switching policy is able to transfer to the protected sensor before significant deviation can occur. This switch allows the vehicles performance to gracefully degrade while avoiding total failure.

\section{Conclusion}
\label{Sec8}

This paper constructed a dynamic watermarking approach for detecting malicious sensor attacks for general LTI systems, and the two main contributions were: to extend dynamic watermarking to general LTI systems under a specific attack model that is more general than replay attacks, and to show that modeling is important for designing watermarking techniques by demonstrating how persistent disturbances can negatively affect the accuracy of dynamic watermarking.  Our approach to resolve this issue was to incorporate a model of the persistent disturbance via the internal model principle.  Future work includes generalizing the attack models that can be detected by our approach.  An additional direction for future work is to study the problem of robust controller design in the regime of when an attack is detected.

\bibliographystyle{IEEEtran}
\bibliography{IEEEabrv,secure}

% biography section
% 
% If you have an EPS/PDF photo (graphicx package needed) extra braces are
% needed around the contents of the optional argument to biography to prevent
% the LaTeX parser from getting confused when it sees the complicated
% \includegraphics command within an optional argument. (You could create
% your own custom macro containing the \includegraphics command to make things
% simpler here.)
%\begin{IEEEbiography}[{\includegraphics[width=1in,height=1.25in,clip,keepaspectratio]{mshell}}]{Michael Shell}
% or if you just want to reserve a space for a photo:
%
%% if you will not have a photo at all:
%\begin{IEEEbiographynophoto}{John Doe}
%Biography text here.
%\end{IEEEbiographynophoto}
%
%% insert where needed to balance the two columns on the last page with
%% biographies
%%\newpage
%
%\begin{IEEEbiographynophoto}{Jane Doe}
%Biography text here.
%\end{IEEEbiographynophoto}

% You can push biographies down or up by placing
% a \vfill before or after them. The appropriate
% use of \vfill depends on what kind of text is
% on the last page and whether or not the columns
% are being equalized.

%\vfill

% Can be used to pull up biographies so that the bottom of the last one
% is flush with the other column.
%\enlargethispage{-5in}

% that's all folks
\end{document}